\documentclass[10pt]{amsart}
\usepackage[latin1]{inputenc}  
\usepackage[T1]{fontenc}

\usepackage{amscd, amsmath, amsthm, amssymb, amsfonts}
\usepackage{graphicx}
\usepackage{mathrsfs}
\usepackage[all]{xy}
\usepackage{verbatim}
\usepackage[pagebackref=false]{hyperref}
\usepackage{xcolor}
\usepackage{braket}
\usepackage{cancel}
\usepackage[normalem]{ulem}
\usepackage[mathscr]{eucal}
\usepackage{enumitem}
\usepackage{float}
\usepackage{changepage}
\usepackage{listings}

\usepackage{pgf,tikz}
\usetikzlibrary{arrows}
\usetikzlibrary{patterns}
\usetikzlibrary{decorations.pathreplacing,calligraphy}

\usepackage[all]{xy}

\usepackage{subcaption}
\usepackage{geometry}
\usepackage{ifthen}

\newcommand{\Z}{\mathbb{Z}}
\newcommand{\Q}{\mathbb{Q}}
\newcommand{\C}{\mathbb{C}}

\renewcommand{\P}{\mathbb{P}}

\newcommand{\GL}{\mathrm{{GL}}}

\newcommand{\Psaut}{\mathrm{Psaut}}
\newcommand{\Aut}{\mathrm{Aut}}

\newcommand{\eps}{\varepsilon}

\newcommand{\id}{\mathrm{id}}

\newcommand{\co}{\mathrm{co}}

\theoremstyle{plain}
\newtheorem{theorem}{Theorem}[section]
\newtheorem{lemma}[theorem]{Lemma}

\newtheorem{proposition}[theorem]{Proposition}
\newtheorem{corollary}[theorem]{Corollary}

\theoremstyle{definition}
\newtheorem{definition}[theorem]{Definition}

\newtheorem{remark}[theorem]{Remark}

\title[Pseudoautomorphisms of $\P^3$ blown-up at very general points]{The pseudoautomorphism group of $\P^3$ blown-up at $8$ very general points}

\author{C\'ecile Gachet}

\begin{document}

\begin{abstract}
We prove that the pseudoautomorphism group of a blow-up of $\P^3$ at $8$ very general points is trivial. We also establish the injectivity of the Coble representation associated to blow-ups of $\P^3$ at $r\ge 8$ general points, answering a question of Dolgachev--Ortland.
\end{abstract}

\maketitle

\section{Introduction}

It is commonly expected that for $r$ large enough, blow-ups of the projective space $\P^n$ at $r$ very general points have no symmetries. Coble \cite{Cob16} famously showed that for $\P^2$, no automorphism arises as soon as $r\ge 9$. In this short note, we focus on $n=3$ and show the following result.

\begin{theorem}\label{thm-main}
Let $X$ denote the blow-up of $\P^3$ at $r=8$ very general points. The pseudoautomorphism group of $X$ is trivial.
\end{theorem}

Proving such a result in dimension greater than $2$ is known to be challenging; see e.g. \cite[Section 1.2]{SX23}. This is due to the lack of understanding of the subgroup of the Cremona group ${\rm Bir}(\P^3)$ generated by ${\rm Aut}(\P^3)$ and standard Cremona transformations, more specifically to issues pointed out in \cite[Remark 1]{Dol11}. To circumvent these issues for $r=8$, we use the nef anticanonical divisor to decompose pseudo-automorphisms into finite sequences of flops; see Section \ref{sec-3}. 

Triviality of the pseudoautomorphism group of a blow-up of $\P^3$ at $r\ge 9$ very general points remains an open question. 
Meanwhile for $r\le 4$, the blow-ups are toric varieties, and for $r$ ranging from $5$ to $7$, they are Mori Dream Spaces: In all these cases, the pseudoautomorphism groups are essentially known by \cite{DO88}.

Along the way, we establish the injectivity of the Coble representation for $r\ge 8$, hereby answering a question of Dolgachev--Ortland; see \cite[Page 130]{DO88}. 

\begin{theorem}\label{thm-r9}
For $r\ge 8$, the Coble representation $\co_{3,r}$ of the Weyl group $W_{3,r}$ on the moduli space of semistable $r$-tuples of points in $\P^3$ is injective.
\end{theorem}

Our strategy of proof is inspired by \cite{Hir88}; see Sections \ref{sec-4}, \ref{sec-5}, and \ref{sec-6}.

Let us finally mention a consequence of Theorem \ref{thm-main} and of \cite{SX23}.

\begin{corollary}\label{cor-cc}
Let $X$ be the blow-up of $\P^3$ at eight very general points and consider a Calabi--Yau pair $(X,\Delta)$. Then the pair $(X,\Delta)$ is not klt and it fails the movable cone conjecture.
\end{corollary}

\medskip

\noindent {\bf Acknowledgements.} I was first made aware of the question resolved in Theorem \ref{thm-main} by Isabel Stenger and Zhixin Xie in Saarbr\"ucken in June 2023, and I am grateful for the many discussions we had about their article \cite{SX23}, and their comments on a draft. I also thank Michel Brion for rebutting an incorrect idea I aired when visiting Grenoble in June 2024. I am partially supported by the projects GAG ANR-24-CE40-3526-01 and the DFG funded TRR 195 - Projektnumber@ 286237555.

\section{Notations and preliminaries}\label{sec-2}

Throughout this paper, we work over an uncountable algebraically closed field $k$ of arbitrary characteristics.
For a smooth projective variety $X$, we denote by $N^1(X)$, respectively $N^1(X)_{\Q}$ the space of numerical equivalence classes of $\Z$, respectively $\Q$-divisors on $X$. 
We define $\Psaut(X)$ to be the group of birational automorphisms of $X$ that are isomorphisms in codimension one.
Pulling back $\Q$-Cartier divisor classes by pseudoautomorphisms induces a natural representation
$$\rho:\Psaut(X)\to\GL(N^1(X)/{\rm tors}),$$
whose image is denoted by $\Psaut^{\ast}(X)$.

For an isomorphism in codimension one $g:X\dashrightarrow Y$ between two normal projective threefolds, we define the
{\it isomorphism open sets} of $g$ to be the maximal Zariski open sets $U\subset X$ and $V\subset Y$ such that the complements $X\setminus U$ and $Y\setminus V$ have pure dimension one and $g$ induces an isomorphism between $U$ and $V$.

\medskip

We say that a set of points in $\P^3$ are {\it linearly independent} if no four of them lie on a common plane. For $r\ge 1$ and ${\rm p}$ an $r$-tuple of distinct, linearly independent points in $\P^3$, we denote by $X_{\rm p}$ the blow-up of $\P^3$ at the center ${\rm p}$. We denote by $\eps_{\rm p}:X_{\rm p}\to \P^3$ the blow-up of ${\rm p}$ in $\P^3$, by $H$ the class of a hyperplane in $\P^3$ and by $E_1,\ldots,E_r$ the exceptional divisors above the points of ${\rm p}$.

We define the following lattice: ${\rm H}_r =\bigoplus_{i=0}^r\Z h_i$ endowed with the symmetric bilinear form
$$(h_i,h_j)=\left\lbrace
\begin{array}{ll}
2\delta_{ij} & \mbox{ if }ij=0,\\
-\delta_{ij} & \mbox{ otherwise}.
\end{array}\right.$$
The corresponding quadratic form is hyperbolic.
Following \cite[Bottom of Page 69]{DO88}, we introduce the strict geometric marking 
$$\varphi_{\rm p}:{\rm H}_r\overset{\sim}{\longrightarrow}N^1(X_{\rm p})/{\rm tors}$$
sending $h_0$ to $\eps_{\rm p}^*H$ and $h_i$ to $E_i$. The induced hyperbolic quadratic form on $N^1(X_{\rm p})/{\rm tor}$ is denoted by $q_{\rm p}$.

It will often be the case that all points of ${\rm p}$ belong to the same smooth quartic curve in $\P^3$: We reserve the notation $C_{\rm p}$ for the curve in this case.

\subsection{The Weyl group \texorpdfstring{$W_{3,r}$}{W3,r}}

\begin{definition}\label{def-weyl}
For $r\ge 5$, we denote by $W_{3,r}$ the Weyl group associated to the root system $T_{2,4,r-4}$, whose Dynkin diagram is depicted in Figure \ref{fig-dynkin}. It comes with a preferred set of involutive generators, which we denote by $\tau_1,\ldots,\tau_{r-1},s$: We set the generators $\tau_i$ to correspond to the vertices of the horizontal chain present in the diagram, from left to right; We set $s$ to correspond to the remaining vertex.
\end{definition}

\begin{figure}[H]
  \begin{tikzpicture}[scale=.4]
    \foreach \x in {0,1,2,3,4,6,7}
    \draw[thick,xshift=\x cm] (\x cm,0) circle (3 mm);
    \foreach \y in {0,1,2,3,6}
    \draw[thick,xshift=\y cm] (\y cm,0) ++(.3 cm, 0) -- +(14 mm,0);
    \draw[thick, loosely dotted] (9 cm,0) -- +(2 cm,0);
    \draw[thick] (6 cm,- 2 cm) circle (3 mm);
    \draw[thick] (6 cm, -3mm) -- +(0, -1.4 cm);
    \draw[thick, decorate, decoration = {calligraphic brace}] (5.7cm, 0.8cm) --  +(8.4cm, 0) node[yshift = 0.4cm, xshift=-1.7cm]{$r-4$ vertices};
  \end{tikzpicture}
  
  \caption{The Dynkin diagram of $T_{2,4,r-4}$}\label{fig-dynkin}
\end{figure}
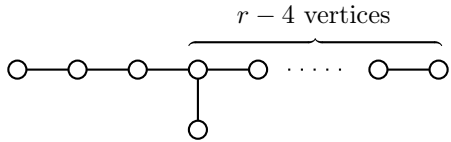

\begin{remark} We mention a few obvious facts.
\begin{enumerate}
\item $\tau_1,\ldots,\tau_{r-1}$ generate a copy of the symmetric group $\mathfrak{S}_r$ in $W_{3,r}$.
\item For small values of $r$, we recover known root systems, namely:
$$T_{2,4,1}= A_5,\quad T_{2,4,2} = D_6,\quad T_{2,4,3} = E_7,\quad T_{2,4,4} = \tilde{E_7}.$$
We recall that $\tilde{E_7}$ is the affine root system based on $E_7$.
\end{enumerate}
\end{remark}

We recall a natural action of $W_{3,r}$ on the hyperbolic lattice ${\rm H}_r$ (see \cite{Dol83,Muk04}, \cite[Page 72, 73]{DO88}, \cite[Section 4.1]{SX23}).

\begin{proposition}\label{prop-whyperbolic}
There is an injective group morphism
$$\pi_r:W_{3,r}\hookrightarrow {\rm Isom}^+({\rm H}_r)$$
sending the generator $\tau_i$ to the hyperbolic reflection relative to $h_i-h_{i+1}$ and $s$ to the hyperbolic reflection relative to $h_0-h_1-h_2-h_3-h_4$.
\end{proposition}

Fixing an $r$-tuple ${\rm p}$ of distinct, linearly independent points in $\P^3$, we thus obtain a faithful representation of $W_{3,r}$ on the N\'eron--Severi space of the variety $X_{\rm p}$, which preserves the quadratic form $q_{\rm p}$:
$$\varphi_{\rm p}\circ\pi_r\circ \varphi_{\rm p}^{-1}:W_{3,r}\hookrightarrow {\rm GL}(N^1(X_{\rm p})/{\rm tors} \, ; q_{{\rm p}}).$$

\subsection{The Coble representation}
The following representation of the Weyl group $W_{3,r}$ was introduced by Kantor, Coble and du Val in \cite{Kan95,Cob16,dV36}, and subsequently studied in \cite{Dol83,Koi88,Hir88, DO88,Muk04}. The next proposition is due to \cite[Section 7, Page 292]{Dol83}; see also \cite[Remark 4.3]{SX23}. 

\begin{proposition}\label{prop-coblerep}
Let $r\ge 4$, and let $U_r$ denote the moduli space of $r$-tuples of distinct, linearly independent points in $\P^3$
There is a representation 
$$\co_{3,r}:W_{3,r}\to{\rm Bir}(U_r)$$
sending $\tau_i$ to the transposition exchanging the $i$-th point with the $(i+1)$-th point of the tuple, for $1\le i\le r-1$, and sending $s$ to the birational map
$$\co_{3,r}(s) : [p_i] \mapsto [p_1,p_2,p_3,p_4,\mathrm{cr}(p_5),\mathrm{cr}(p_6),\ldots,\mathrm{cr}(p_r)],$$ 
where $\mathrm{cr}$ is the standard Cremona transformation of $\P^3$ centered at the four points $p_1,p_2,p_3,p_4$. We call $\co_{3,r}$ the \emph{Coble representation} of $W_{3,r}$.
\end{proposition}

\begin{remark}\label{rem-Vr}
To prove that the action of $W_{3,r}$ is generically well-defined on a subset $V_r$ of $U_r$ and preserves $V_r$, it suffices to check it for each of the $r+1$ standard generators of $W_{3,r}$. 
As an application, consider the closure of the set 
$$V_r:= \{{\rm p}\in U_r \mid \mbox{there is a pencil of quadrics containing }{\rm p}\}.$$
It is clearly preserved by the action of the symmetric subgroup of $W_{3,r}$. Since the standard Cremona transformation preserves the linear system of quadrics through its four center points, the action by the last generator $s\in W_{3,r}$ also preserves the closure of $V_r$. 
So the action of $W_{3,r}$ by $\co_{3,r}$ is generically well-defined on that particular $V_r$ and preserves it.
\end{remark}

\section{Pseudoautomorphisms preserve the quadratic form \texorpdfstring{$q_{\rm p}$}{qp} when \texorpdfstring{$r\le 8$}{r>=8}}\label{sec-3}

The main result of this section is the following proposition.

\begin{proposition}\label{prop-psautqp}
Let ${\rm p}$ be a very general $r$-tuple of points in $\P^3$ with $r\le 8$. The action by $\Psaut^*(X_{\rm p})$ preserves $q_{\rm p}$.
\end{proposition}

Before proving Proposition \ref{prop-psautqp}, we prove a lemma.

\begin{lemma}\label{lem-flops}
Let ${\rm p}$ be a very general $r$-tuple of points in $\P^3$ with $r\le 8$.
For any finite sequence of flops $\alpha: X_{\rm p}\dashrightarrow Y$, the very general member $S$ of the linear system $\left\vert -\frac{1}{2}K_{X_{\rm p}}\right\vert$ is such, that the restriction
$\alpha|_S:S\dashrightarrow \alpha_*(S)$ is an isomorphism.
\end{lemma}

\begin{proof}[Proof of Lemma \ref{lem-flops}]
We argue by induction of the minimal number $n$ of flops needed to factorize $\alpha$. For $n=0$, we have an isomorphism, and the claim clearly holds. Assume that the claim is known for any finite sequence of $n$ flops, and let $\alpha$ be a sequence of $n+1$ flops. We decompose $\alpha = \alpha'\circ \varphi$, where $\alpha':X_{\rm p}\dashrightarrow Y'$ is a sequence of $n$ flops and $\varphi:Y'\dashrightarrow Y$ is a single flop. Let $S$ be a smooth, very general member of the linear system $\left\vert -\frac{1}{2}K_{X_{\rm p}}\right\vert$ such that
\begin{itemize}
\item $S$ is smooth;
\item if $r\le 7$, $S$ is a del Pezzo surface, and if $r=8$, $S$ contains no $(-2)$-curve (this condition is very general by \cite[Lemma 5, Proof of Lemma 6]{LO16});
\item the restrictionn $\alpha'|_S:S\dashrightarrow \alpha'_*(S)$ is an isomorphism.
\end{itemize}
In particular, the surface $S':= \alpha'_*(S)$ is smooth, iromorphic to $S$ and contained in the smooth locus of $Y'$.

By definition, the flop $\varphi$ is an isomorphism outside of a finite union of $K$-trivial smooth rational curves. Let $C$ be such a curve in $Y'$: Then $S'\cdot C = 0$ holds. If $C$ is contained in $S'$, then the adjunction yields that $-K_{S'}\cdot C = 0$. If $r\le 7$, this contradicts the fact that $S'$ is a del Pezzo surface, and if $r=8$, it contradicts the fact that $S$ contains no $(-2)$-curve by \cite[Lemma 4]{LO16}. So $C$ is disjoint from $S'$, and thus $S'$ is contained in the isomorphism open set of the flop $\varphi$.
\end{proof}

We can now prove Proposition \ref{prop-psautqp}.

\begin{proof}[Proof of Proposition \ref{prop-psautqp}]
One notices that for any divisor $D\in N^1(X_{\rm p})$,
$$q_{\rm p}(D) = D^2\cdot\left(-\frac{1}{2}K_{X_{\rm p}}\right).$$
Let $g\in\Psaut(X_{\rm p})$. By \cite[Corollary 3.4]{SX23}, one can decompose $g$ as a finite sequence of flops.
By Lemma \ref{lem-flops}, we conclude that
$$q_{\rm p}(D) = (D|_S)^2 = \left[(g^*D)|_{g^{-1}_*S}\right]^2 = (g^*D)^2\cdot \left(-\frac{1}{2}K_{X_{\rm p}}\right) = q_{\rm p}(g^*D),$$
for a very general $S\in\left\vert-\frac{1}{2}K_{X_{\rm p}}\right\vert $ and for $D\in N^1(X_{\rm p})$.
\end{proof}

We conclude this subsection with a consequence of Lemma \ref{lem-flops}.

\begin{corollary}\label{cor-definedbasecurve}
Let ${\rm p}$ be a very general $8$-tuple of points in $\P^3$. For any $g\in\Psaut(X_{\rm p})$, the isomorphism open sets of $g$ both contain the curve $C_{\rm p}$
\end{corollary}

\begin{proof}
Since $r=8$, the half-anticanonical linear system $\left\vert -\frac{1}{2}K_{X_{\rm p}}\right\vert$ is a pencil. Its base locus is $C_{\rm p}$. By \cite[Corollary 3.4]{SX23}, the pseudoautomorphism $g$ is a finite sequence of flops, so Lemma \ref{lem-flops} applies and we can find a very general member $S$ of $\left\vert -\frac{1}{2}K_{X_{\rm p}}\right\vert$ that is contained in the isomorphism open sets of $g$. So is $C_{\rm p}\subset S$.
\end{proof}

\section{Restricting the action to the quartic curve \texorpdfstring{$C_{\rm p}$}{Cp}}\label{sec-4}

In this section, we work with $r\ge 8$ and ${\rm p}$ an $r$-tuple of points of $\P^3$ that is very general among $r$-tuples supported on a pencil of quadrics. In the notation of Remark \ref{rem-Vr}, ${\rm p}$ represents a very general point in $V_r\subset U_r$. For $r=8$, we have $V_8=U_8$, thus the condition is fulfilled by a very general $r$-tuple. For $r\ge 9$, it means that the support of ${\rm p}$ is contained in a very general quartic curve $C_{\rm p}\subset \P^3$. This curve is uniquely determined by ${\rm p}$. As a smooth genus one curve, $C_{\rm p}$ is moreover very general in moduli, and does not have complex multiplication.

The following lemma generalizes \cite[2.4. Lemme]{Hir88} from blow-ups of $\P^2$ in nine or more points to blow-ups of $\P^3$ in eight or more points. It should also relate to \cite[Lemma 5]{LO16}.

\begin{lemma}\label{lem-trinj}
Let $r\ge 8$ and ${\rm p}$ be an $r$-tuple of points that is very general among those contained in a common quartic curve $C_{\rm p}$ in $\P^3$. The restriction map
$${\rm tr}:{\rm Pic}(X_{\rm p})\to {\rm Pic}(C_{\rm p})$$
is injective.
\end{lemma}

\begin{proof}
We work with divisors modulo linear equivalence. Let us denote by $H$ the class of a hyperplane in $\P^3$.
Let $D$ be a divisor on $X_{\rm p}$ with ${\rm tr}(D)\sim 0$ and write
$$D\equiv 2d_0 H - \sum_{i=1}^r d_iE_i.$$
Without loss of generality, we can assume that $d_0 \ge 0$. We write $D=A-B$ with $A$ and $B$ both ample divisors.
Note that ${\eps_{\rm p}}_*A$ and ${\eps_{\rm p}}_*B$ define sections of the line bundles $\mathcal{O}_{\P^3}(2d_0+b_0)$, respectively $\mathcal{O}_{\P^3}(b_0)$, passing throuph the points ${\rm p}_i$ with multiplicities $d_i+b_i$ respectively $b_i$, for some integers $b_0,\ldots,b_8>0$. In particular, we have
$$0 = {\rm tr}(D) = {\eps_{\rm p}}_*(A-B)|_{C_{\rm p}} = \sum_{i=1}^r d_i {\rm p}_i + (2d_0H)|_{C_{\rm p}}.$$
Note that $(2H)|_{C_{\rm p}}$ is linearly equivalent to the sum of any seven of the ${\rm p}_i$ and of the eighth base point of the net of quadrics passing through them. Since ${\rm p}$ is very general, this enforces $d_1=\ldots=d_r=0$, thus $d_0=0$ and $D=0$, as wished.
\end{proof}

The next lemma derives from a very classical argument; see \cite{Cob16}, \cite[Proposition 8]{Giz81}, \cite[2.3. Proposition]{Hir88}.

\begin{lemma}\label{lem-gactingrlarge}
Let $r\ge 8$ and ${\rm p}$ be an $r$-tuple of points that is very general among those contained in a common quartic curve $C_{\rm p}$ in $\P^3$. Let $g\in\Psaut(X_{\rm p})$ such that the curve $C_{\rm p}$ in $X_{\rm p}$ is contained in the isomorphism open sets of $g$ and that the pullback $g^*$ preserves the quadratic form $q_{\rm p}$. Then there exist $\sigma\in\{\pm 1\}$ and $L\in {\rm Pic}(X_{\rm p})$ such that 
\begin{align*}
g^*E_i 
&= \sigma E_i + L\mbox{ for every }1\le i\le 8\mbox{ and}\\
g^*\eps_{\rm p}^*H 
&= \sigma \eps_{\rm p}^*H + 4L.
\end{align*}
\end{lemma}

\begin{proof}[Proof of Lemma \ref{lem-gactingrlarge}]
Since $C_{\rm p}$ is entirely contained in the isomorphism open sets of $g$, pulling back by $g$ transforms the exceptional divisors $E_i$ into prime divisors $F_i$ which still satisfy $F_i\cdot C_{\rm p} = 1$ for all $1\le i\le r$. The unique intersection point of $F_i$ with $C_{\rm p}$ is then the image of ${\rm p}_i$ by the automorphism $g^{-1}|_{C_{\rm p}}$ of the smooth curve $C_{\rm p}$ of genus one. Since $C_{\rm p}$ is a very general elliptic curve, we can write 
$$g^{-1}|_{C_{\rm p}}\in\Aut(C_{\rm p}) = C_{\rm p}\rtimes \Z/2\Z,$$
which acts by translations and inversion with respect to a fixed origin point, say ${\rm p}_1$. Let $\sigma$ be $1$ if $g^{-1}|_{C_{\rm p}}$ is a translation and $-1$ otherwise. Let $t\in C_{\rm p}$ be the image of ${\rm p}_1$ by $g^{-1}|_{C_{\rm p}}$.
This shows two things:
\begin{itemize}
\item that $t - \sigma {\rm p}_1  = (F_i -\sigma E_i)|_{C_{\rm p}}$ in ${\rm Pic}(C_{\rm p})$, which does not depend on $i$;
\item that $4(t- \sigma {\rm p}_1) = (g^*\eps_{\rm p}^*H - \sigma \eps_{\rm p}^*H)|_{C_{\rm p}}$.
\end{itemize}
By Lemma \ref{lem-trinj}, we deduce that the divisor $F_i -\sigma E_i$ in ${\rm Pic}(X_{\rm p})$ does not depend on $i$ either. That is the divisor $L$ we are seeking after. Applying Lemma \ref{lem-trinj} to the divisors $4L$ and $g^*\eps_{\rm p}^*H - \sigma \eps_{\rm p}^*H$ concludes.
\end{proof}

\section{On certain isometries preserving the effective cone}\label{sec-5}

We keep following the argument of \cite{Hir88} and \cite{DO88} for this last statement, as the more hands-on approach of \cite{Koi88} seems harder to generalize. 

\begin{lemma}\label{lem-Lzerosone}
Let $r\ge 8$ and ${\rm p}$ be an $r$-tuple of points that is very general among those contained in a common quartic curve in $\P^3$. Let $\iota$ be an isometry of the lattice $N^1(X_{\rm p})$ with $q_{\rm p}$, that preserves the effective cone and the anticanonical class. Assume that there exist $\sigma\in\{\pm 1\}$ and $L\in {\rm Pic}(X_{\rm p})$ such that \begin{align*}
\iota(E_i)
&= \sigma E_i + L\mbox{ for every }1\le i\le r\mbox{ and}\\
\iota(\eps_{\rm p}^*H)
&= \sigma \eps_{\rm p}^*H + 4L.
\end{align*}
Then $L=0$ and $\sigma = 1$, i.e., $\iota$ is the identity.
\end{lemma}

\begin{proof}
Since $\iota$ preserves the quadratic form $q:=q_{\rm p}$, we have
\begin{equation}\tag{$\ast$}
q(E_i,L) = -\alpha\mbox{ for every }1\le i\le r\mbox{ and }q(\eps_{\rm p}^*H,L) = 4\alpha,
\end{equation}
where $\alpha$ denotes the scalar $\frac{q(L)}{2\sigma}$.
Since $q$ is non-degenerate, this and an easy computation imply that $L$ is numerically equivalent to the divisor class: $-\frac{\alpha}{2}K_{X_{\rm p}}$. 

Let us first assume that $r=8$. Then we have $q\left(-\frac{1}{2}K_{X_{\rm p}}\right)=0$. Therefore, $q(L)=0$, that is $\alpha = 0$, and as a result $L=0$. Since $\iota$ preserves the effective cone, which is non-degenerate, the sum $\iota(E_i)+E_i$ cannot be zero, and thus $\sigma =1$.

Let us now assume that $r\ge 9$. Using that $\iota$ preserves both $q$ and the anticanonical class on the left handside and Identity ($\ast$) on the right handside, we note that
$$1 -\sigma 
= q\left(\iota(E_1) - \sigma E_1, -\frac{1}{2}K_{X_{\rm p}}\right) 
= q\left(L, -\frac{1}{2}K_{X_{\rm p}}\right)
= \alpha(8-r).$$
In particular, if $\sigma = 1$, then $\alpha=0$ and thus $L=0$, as wished.

Let us finally assume by contradiction that $\sigma = -1$.
Then $\alpha$ is negative. 
Then also $L=\iota(E_i)+E_i$ is an effective class, and so are its positive multiples, such as the canonical class $K_{X_{\rm p}}$. However, the curve class $(\eps_{\rm p}^*H)^2$ is strongly movable, and
$$K_{X_{\rm p}}\cdot (\eps_{\rm p}^*H)^2 =-2\, q(\eps_{\rm p}^*H) = -4 < 0,$$
a contradiction.
\end{proof}

\section{Proof of the main theorems}\label{sec-6}

We start with a simple, yet important fact.

\begin{lemma}\label{lem-psaut*}
Let $X$ be a blow-up of $\P^3$ at $r$ points, five of which form a general $5$-tuple in $\P^3$. Then the representation $\rho:\Psaut(X)\to\GL(N^1(X)/{\rm tors})$ is faithful.
\end{lemma}

\begin{proof}
An application of the negativity lemma (see for instance \cite[Lemma 4.2]{GLSW24}) shows that $\ker \pi\subset\Aut(X)$. Any element $g\in\ker\pi$ notably fixes the numerical classes of the exceptional divisors $E_1,\ldots,E_r$ of the blow-up map $\eps:X\to\P^3$. Since $E_i$ is not numerically equivalent to any effective divisor other than itself, $g$ descends under $\eps$ to an automorphism $\gamma\in\mathrm{PGL}(4,\C)$ fixing the $r$ blown-up points and in particular a general $5$-tuple of points. Thus, $\ker\pi = \{\id_X\}$.
\end{proof}

We now prove our main theorems.

\begin{proof}[Proof of Theorem \ref{thm-main}]
Recall that $X$ denotes the blow-up of $\P^3$ at eight very general points. Let $g\in\Psaut(X)$. By Corollary \ref{cor-definedbasecurve} and Proposition \ref{prop-psautqp}, Lemma \ref{lem-gactingrlarge} applies to $g$. Thus and by Proposition \ref{prop-psautqp} again, the isometry $\iota:=g^*$ satisfies the assumptions of Lemma \ref{lem-Lzerosone}, therefore
$$g^*E_i=E_i
\mbox{ for all $1\le i\le r$ and }
g^*\eps^*H = \eps^*H,$$
where $\eps=X\to\P^3$ denotes the blow-up map and $E_i$ its exceptional divisors.
Hence $g$ belongs to the kernel of the linear representation
$$\rho:\Psaut(X)\to{\rm GL}(N^1(X)/{\rm tors}),$$
and Lemma \ref{lem-psaut*} concludes that $g$ is trivial.
\end{proof}

\begin{proof}[Proof of Theorem \ref{thm-r9}]
Note that \cite[Theorem 5, Page 99]{DO88} and Theorem \ref{thm-main} immediately imply Theorem \ref{thm-r9} for $r=8$.

We now prove the theorem for $r\ge 9$. Let $w\in W_{3,r}$ with $\co_{3,r}(w)$ trivial. Let ${\rm p}$ be an $r$-tuple of points that is very general among those supported on a quartic curve $C_{\rm p}$ in $\P^3$. By Remark \ref{rem-Vr} and \cite[Page 99, Lemma 2]{DO88} and obtain a pseudoautomorphism $g$ of $X_{\rm p}$ satisfying
$$\varphi_{\rm p}\circ \pi_r(w)\circ\varphi_{\rm p}^{-1}=g^*.$$
In particular, the pullback $g^*$ preserves the quadratic form $q_{\rm p}$. 
We claim that the curve $C_{\rm p}$ is contained in the isomorphism open sets of $g$. Once that claim is established, we can apply Lemmas \ref{lem-gactingrlarge} and \ref{lem-Lzerosone} to derive that the pullback $g^*$ is trivial, and conclude by faithfulness of the representation $\varphi_{\rm p}\circ \pi_r\circ\varphi_{\rm p}^{^-1}$ that $w\in W$ is trivial too.

Let us prove the claim. In fact by Remark \ref{rem-Vr}, it makes sense to prove more generally that for any element $v\in W_{3,r}$, denoting ${\rm q}:= \co_{3,r}(v)({\rm  p})$, the isomorphism in codimension one
$$g:X_{\rm q}\dashrightarrow X_{\rm  p}$$
induced as in \cite[Theorem 1]{Muk04}, \cite[Page 86, Proposition 1]{DO88} 
has its isomorphism open sets contain the curves $C_{\rm q}$ and $C_{\rm  p}$ respectively.
We proceed by induction on the minimal number $k$ of occurences of the generator $s$ necessary to write out $v\in W_{3,r}$. If none is needed, then $g$ is the identity and the result holds. Fix $k\ge 1$ and assume that the result holds for $k-1$. Let $v\in W_{3,r}$ be an element optimally written with exactly $k$ occurences of $s$. Using that $W_{3,r}$ is a Coxeter group, we rewrite $v=us$, where $u\in W_{3,r}$ can be written with strictly fewer occurences of $s$. Consider the isomorphisms in codimension one induced by $v$ and $u$, namely
$$g:X_{\rm q}\dashrightarrow X_{\rm  p}\mbox{ and }
h:X_{\rm q}\dashrightarrow X_{\co_{3,r}(s)({\rm p})}$$
respectively, and let $c:X_{\co_{3,r}(s)({\rm p})}\dashrightarrow X_{\rm  p}$ be the lift of the standard Cremona transformation of $\P^3$ centered at the first four points of ${\rm p}$.
Note that 
$g= c\circ h.$
The isomorphism open sets of $c$ are known to be the complements of the strict transforms of the six lines through ${\rm p}_1,\ldots, {\rm p}_4$: In particular, they contain the two curves $C_{\rm p}$ and $C_{\co_{3,r}(s)({\rm p})}$. By the induction hypothesis, the isomorphism open sets of $h$ contain $C_{\co_{3,r}(s)({\rm p})}$ and $C_{\rm q}$. This shows that $g$ indeed contains $C_{\rm q}$ and $C_{\rm  p}$ in its respective isomorphism open sets.
\end{proof}

\begin{proof}[Proof of Corollary \ref{cor-cc}]
Let $C$ be the curve that is the base locus of $\left\vert-\frac{1}{2}K_X\right\vert$. The divisor $-K_X$ is not semiample \cite{LO16}. Thus, for $m\ge 1$ and for $D\in |-mK_X|$, the curve $C$ is not disjoint from $D$: But $-K_X\cdot C = 0$, so $C$ is contained in $D$. This shows that $C$ is in the base-locus of the linear system $|-mK_X|$ too, hence no Calabi-Yau pair $(X,\Delta)$ is klt along $C$.

By \cite[Lemma 7.1]{SX23}, the movable effective cone ${\rm Mov}^{\rm e}(X)$ is not rational polyhedral. By Theorem \ref{thm-main}, the group $\Psaut(X)$ is trivial, and so is the subgroup $\Psaut(X,\Delta)$ for any pair $(X,\Delta)$. The pair $(X,\Delta)$ clearly fails the movable cone conjecture.
\end{proof}

\bigskip

\bibliographystyle{alpha}
\bibliography{bibPsautP3.bib}

\end{document}